\documentclass[11pt,a4paper]{amsart}
\usepackage{amsmath, amsthm, amscd, amsfonts}

\newtheorem{theorem}{Theorem}[section]
\newtheorem{lemma}[theorem]{Lemma}
\newtheorem{proposition}[theorem]{Proposition}

\theoremstyle{definition}
\newtheorem{definition}[theorem]{Definition}
\newtheorem{example}[theorem]{Example}

\newtheorem{remark}[theorem]{Remark}
\numberwithin{equation}{section}

\begin{document}

\title[Gabor Tight Fusion Frames]{Gabor Tight Fusion Frames: Construction and Applications in Signal Retrieval Modulo Phase}

 
\author[Mohammadpour]{Mozhgan Mohammadpour}
\address{Department of Pure Mathematics, Faculty of Mathematical sciences, Ferdowsi University of Mashhad, Iran}
\email{mozhganmohammadpour@gmail.com}

 \author[Tuomanen]{Brian Tuomanen}
\address{Department of Mathematics, University
of Missouri, Columbia, MO 65211-4100, USA}
\email{btuomanen@outlook.com}

\author[Kamyabi Gol]{Rajab Ali Kamyabi Gol}
\address{Department of Pure Mathematics, Faculty of Mathematical sciences, Ferdowsi University of Mashhad, Iran}
\email{kamyabi@um.ac.ir}

\thanks{The second author was supported by NSF ATD 1321779.}

\begin{abstract}
Hilbert space fusion frames are a natural extension of Hilbert space frames, extending the notion from a set of vectors in a Hilbert space to a set of subspaces of a Hilbert space with analogous notions of overcompleteness and boundedness.  As tight frames are a very important topic within standard frame theory, tight fusion frames are similarly important;  however, only trivial examples of tight fusion frames are hitherto known. In this paper, we apply ideas from Gabor analysis to demonstrate a non-trivial construction of tight fusion frames by using the fact that a fusion frame for a finite dimensional Hilbert space $\mathcal{H}$ with $M$ subspaces is a frame for the finite dimensional Hilbert space $\mathcal{H}^M$. We then use this construction to further show their applicability in some cases for the retrieval of signals modulo phase.
\end{abstract} 

\maketitle

\section{\bf{Introduction}} 

Fusion frame theory has recently garnered great interest among researchers who work in signal processing.  Fusion frames extend the notion of a frame (i.e., an overcomplete set of vectors) within a Hilbert space $\mathcal{H}$ to a collection of subspaces $\{W_i\}_{i \in I}$ (with orthogonal projections $\{P_i \}_{i \in I}$) in $\mathcal{H}$.  This concept was originally introduced by Kutyniok and Casazza in \cite{FOS}.  

A tight fusion frame  is one such that we have the identity $\sum_{i\in I} P_i=CI_{N\times N}$, i.e., the sum of the  projections is a multiple of the identity. Such tight fusion frames are of interest for two reasons.  First, they guarantee a very simple reconstruction of a signal; and second, tight fusion frames are robust against noise \cite{TFF} and  also remain robust against a one-erasure subspace when the rank of projections are equal to each other \cite{RobustDim}. 

On the other hand, phaseless reconstruction is a field that has gathered interest in the mathematical community in the last decade.  Phaseless reconstruction (or equivalently, \textit{phase retrieval}) is defined as the recovery of a signal modulo phase from the absolute values of fusion frame measurement coefficients arising from a fusion frame. This is known to have applications to a disparate array of other scientific and applied disciplines, including X-ray crystallography \cite{Drenth}, speech recognition \cite{Becchetti, Rabiner, Proakis}, and quantum state tomography \cite{Renes}, where the recorded phase information of a signal is lost or distorted.  

In the case of phase retrieval, the signal must be recovered from coefficients of dimension higher than one.  Here, in the context of fusion frames, the problem is to recover $x\in \mathcal{H}_M$ ``up to phase" from the measurements $\{\Vert P_i x \Vert\}_{i=1}^N$.

In this paper we demonstrate a new method for the construction of tight fusion frames. There are hithero few examples of tight fusion frames except trivial ones made up of orthogonal subspaces, so we believe this is a relevant and interesting advance.  Moreover, there are few examples of phase retrieval fusion frames. In this paper, we present a condition that makes this structure allow phase retrieval. 

This article is organized as follows: Section 2 starts with preliminaries about tight fusion frames and phase retrievability of fusion frames. Section 3 is devoted to a brief summary of Gabor frames which is used to construct the tight fusion frames. In section 4, we explain our method to construct tight fusion frames. Section 5 focuses on finding conditions that makes our tight fusion frame allow phase retrieval, and our conclusion is in section 6.

\section{\bf{ Preliminaries And Notation }}

A fusion frame is defined as follows:

\begin{definition}
\label{FusionFrameDef}
Consider a Hilbert space $\mathcal{H}$, with a collection of subspaces $\{W_i\}_{i \in I}$ and an associated set of positive weights $\{ \nu_i \}_{i \in I}$ .  We likewise denote the associated orthogonal projections $P_i : \mathcal{H} \mapsto W_i$.  Then we call $\{W_i\}_{i\in I}$ a \textbf{fusion frame} if there are positive constants $0 < A \leq B < \infty$ such that for any $x \in \mathcal{H}$ we have the following:

$$ A \| x \|^2 \le \sum_{i \in I}\| P_i x \|^2 \le B \| x \|^2 $$

\end{definition}

\begin{definition}
A \textbf{ tight fusion frame} is a fusion frame as in \ref{FusionFrameDef} where $A=B$ for all $i \in I$.  That is to say, we have the following for any $x \in \mathcal{H}$:  

$$ \sum_{i \in I} \| P_i x \|^2 = A \| x \|^2 $$

Or, equivalently:

\[
AI=\sum_{i=1}^N P_i 
\]
\end{definition}

Now, consider an orthonormal basis for the range of $P_i$, that is $\{e_{i, \ell}\}_{i=1}^n$.  We know that:
\[
P_ix=\sum_{\ell=1}^n \langle x, e_{i, \ell} \rangle e_{i,\ell}
\]
for all $\mathbf{x}\in \mathbb{C}^N$. Summing these equations over $i=1,\cdots,N$ together 
\[
Ax=\sum_{i=1}^N P_i\mathbf{x}=\sum_{i=1}^N \sum_{ \ell =1}^n \langle \mathbf{x}, e_{i, \ell} \rangle e_{i, \ell}
\]

One can recover the signal modulo phase from fusion frame measurements. In this senario, consider we are given subspaces $\{W_i\}_{i=1}^N$ of $M$-dimensional Hilbert space $\mathcal{H}_M$ and orthogonal projections $P_i:\mathcal{H}_M\rightarrow W_i$. We want to recover any $\mathbf{x}\in \mathcal{H}_M$ (up to a global phase factor) from the fusion frame measurements $\{\Vert P_i\mathbf{x}\Vert \}_{i=1}^N$. To fix notation, denote $\mathbb{T} = \{c\in \mathbb{C}; \vert c\vert =1\}$. The measurement process is then given by the map:
\[
\mathcal{A}:{\mathbb{C}^M}/{\mathbb{T}}\rightarrow \mathbb{C}^N, \quad \mathcal{A}\mathbf{x}\left(n\right)=\Vert P_n\mathbf{x}\Vert
\]

We say $\{W_i\}_{i=1}^N$ \textbf{allows phaseless reconstruction} or \textbf{allows phase retrieval} if  $\mathcal{A}$ is injective;  we call a frame (or fusion frame) with this property a \textbf{phase retrieval frame}. In the case where $\dim W_i=1$ for $i=1,\cdots, N$, the problem will be referred to as the classical phaseless reconstruction problem. In section $4$, we will provide a novel structure of tight fusion frames where under particular conditions, will allow phaseless reconstruction.

\section{ \bf{Gabor Frames For $\mathbb{C}^N$}} 
In this section, we provide a brief summary of Gabor frames which is used to construct our tight fusion frames.
We index the components of a vector $\mathbf{x}\in \mathbb{C}^N$ by
$\{0, 1, \cdots, N-1\}$, i.e., the cyclic group
$\mathbb{Z}_{N}$.  We will write $\mathbf{x}\left( k \right)$ 
instead of $\mathbf{x}_k$ to avoid algebraic operations on
indices.  

The discrete Fourier transform is basic in Gabor analysis and is defined as
\[ 
\mathcal{F} \mathbf{x} \left( m \right) = \hat{\mathbf{x}}\left( m \right) = \sum_{n=0}^{N-1} \mathbf{x}\left( n\right) e^{-2\pi i m\frac{n}{N}}. 
\] 

The most important properties of the Fourier transform are the Fourier
inversion formula and the Parseval formula
\cite{FiniteFrames}. The inversion formula shows that
any $\mathbf{x}$ can be written as a linear combination of harmonics.
This means the normalized harmonics $\{\frac{1}{\sqrt{N}} e^{2\pi i m
  \frac{\left(. \right)}{N}}\}_{m=0}^{N-1}$ form an orthonormal basis
of $\mathbb{C}^N$ and hence we have 
\[ 
\mathbf{x}= \frac{1}{N} \sum_{m=0}^{N-1} \hat{\mathbf{x}} \left( m \right) e^{2\pi i m \frac{n}{N}} \quad \mathbf{x} \in \mathbb{C}^N. 
\] 
Moreover, the Parseval formula states
\[ 
\langle \mathbf{x}, \mathbf{y} \rangle = \frac{1}{N} \langle \hat{\mathbf{x}}, \hat{\mathbf{y}} \rangle \quad \mathbf{x}, \mathbf{y} \in \mathbb{C}^N,
\] 
which results in 
\[ 
\sum_{n=0}^{N-1} \vert \mathbf{x} \left( n \right) \vert ^2 = \frac{1}{N} \sum_{m=0}^{N-1} \vert \hat{\mathbf{x}} \left( m \right) \vert ^2,
\] 
where $\vert\mathbf{x}\left(n \right) \vert ^2$ quantifies the energy
of the signal $\mathbf{x}$ at time $n$, and the Fourier coefficients
$\hat{\mathbf{x}}\left(m \right)$ indicates that the harmonic $e^{2\pi
  i m\frac{\left(.\right)}{N}}$ contributes energy $\frac{1}{N}\vert
\hat{\mathbf{x}}\left(m \right) \vert^2$ to $\mathbf{x}$.

Gabor analysis concerns the interplay of the Fourier transform, translation operators, and modulation operators. The cyclic translation operator $T: \mathbb{C}^N \rightarrow \mathbb{C}^N$ is given by
\[ 
T\mathbf{x}= T \left( \mathbf{x}\left(0 \right), \cdots, \mathbf{x} \left( N-1 \right) \right)^t= \left( \mathbf{x}\left( N-1\right) , \mathbf{x} \left( 0 \right) , \cdots, \mathbf{x} \left( N-2 \right) \right)^t .
\] 
The translation $T_k$ is given by
\[ 
T_k \mathbf{x} \left(n\right)= T^k \mathbf{x}\left( n \right)=\mathbf{x}\left(n-k\right) .
\] 
The operator $T_k$ alters the position of the entries of
$\mathbf{x}$. Note that $n-k$ is achieved modulo $N$. 
The modulation operator $M_\ell:\mathbf{C}^N\rightarrow \mathbb{C}^N$ is
given by  
\[ 
M_\ell\mathbf{x}= \left( e^{-2\pi i \ell\frac{0}{N}} \mathbf{x}\left( 0 \right), e^{-2\pi i \ell\frac{1}{N}} \mathbf{x}\left(1\right), \cdots , e^{-2\pi i \ell\frac{N-1}{N}} \mathbf{x}\left( N-1 \right) \right)^t .
\] 
Modulation operators are implemented as the pointwise product of the
vector with harmonics $e^{-2\pi i \ell\frac{.}{N}}$. 

Translation and modulation operators are referred to as time-shift and frequency shift
operators. The time-frequency shift operator $\pi \left( k,\ell \right)$
is the combination of translation operators and modulation operators:
\[ 
\pi \left(k,\ell \right): \mathbb{C}^N\rightarrow \mathbb{C}^N \quad \pi \left( k,\ell\right) \mathbf{x}= M_\ell T_k \mathbf{x}.
\] 
Hence, the short time-Fourier transform $V_{\phi}:
\mathbb{C}^N\rightarrow \mathbb{C}^{N\times N}$ with respect to the window
$\phi \in \mathbb{C}^N$ can be written as
\[ 
V_{\phi}\mathbf{x} \left( k,\ell \right) = \langle \mathbf{x}, \pi \left( k,\ell \right) \phi \rangle= \sum_{n=0}^{N-1} \mathbf{x}\left(n \right) \overline{\phi\left(n-k\right)} e^{-2\pi i \ell\frac{n}{N}} \quad \mathbf{x} \in \mathbb{C}^N. 
\] 
The short time-Fourier transform generally uses a window function
$\phi$, supported at neighborhood of zero that is translated by
$k$. Hence, the pointwise product with $\mathbf{x}$ selects a portion
of $\mathbf{x}$ centered at $k$, and this portion is analyzed using a
Fourier transform. The inversion formula for the short time-Fourier
transform is \cite{FiniteFrames}
\[ 
\begin{split} 
\mathbf{x}\left( n \right)& = \frac{1}{N\Vert\phi \Vert_2^2} \sum_{k=0}^{N-1} \sum _{\ell=0}^{N-1} V_{\phi}\mathbf{x}\left(k,\ell\right) \phi \left(n-k \right) e^{-2\pi i \ell \frac{n}{N}} \\ 
&=\frac{1}{N\Vert\phi \Vert_2^2} \sum_{k=0}^{N-1} \sum _{\ell=0}^{N-1} \langle \mathbf{x}, \pi \left(k,\ell\right)\phi \rangle\pi\left(k,\ell\right)\phi\left(n\right) \quad \mathbf{x} \in \mathbb{C}^N .
\end{split} 
\] 
So it can be easily seen that for all $\phi\neq 0$, the system is a $N\Vert \phi \Vert ^2$ tight Gabor frame.

\section{\bf{Gabor Fusion Frame For $\mathbb{C}^N$}}
In this section, we show our method to construct Gabor tight fusion frames. In fact, we show that Gabor fusion frame for $\mathbb{C}^N$ is a Gabor frame for $\mathbb{C}^{N\times N}$ where the signal is coming from the subspce $\mathbb{C}^N\subset \mathbb{C}^{N\times N}$. The key idea is to start with a general approach for the construction of tight fusion frames, which has certain conditions that must be satisfied.  We then show that these conditions are indeed satisfied using methods from the Gabor frame theory.

We begin by showing the following proposition, which is the generalization of our approach with certain conditions:

\begin{proposition}\label{main1}
Consider a collection of frame sequences $\{\{f_{ij}\}_{i=1}^L\}_{j=1}^M$ within the finite dimensional Hilbert space $\mathbb{C}^N$, and denote $W_i:=span\{f_{ij}\}_{j=1}^M$. Suppose there exists an index $i_0$ such that $ \{f_{i_0j}\}_{j=1}^M$ is a $B$-tight frame for $W_{i_0}$  and also a set of coisometry operators $\{\mathcal{U}_i\}_{i=1}^L$ from $\mathbb{C}^N$ to $\mathbb{C}^N$ such that for each $j=1,...,M$, we have 
\[
\{f_{ij}\}_{i=1}^L =\{\mathcal{U}_i f_{i_0j}\}_{i=1}^L.
\]  
Furthermore, if the set $\{f_{ij}\}_{i=1}^L$ is an $A_j$-tight frame in $\mathbb{C}^N$ for every $j=1,\cdots,M$.\\
Then we will have that $\{\left(W_i,1\right)\}_{i=1}^L$ is a tight fusion frame.

\end{proposition}

\begin{proof}
Consider $\mathbf{x}\in W_i$. The set $\{\mathcal{U}_if_{i_0j}\}_{j=1}^M$ is a $B$-tight frame for $W_i$ over $i=1,\cdots,L$, because 
\[
\begin{split}
\sum_{j=1}^M \vert\langle \mathbf{x}, \mathcal{U}_if_{i_0j} \rangle \vert^2&=\sum_{j=1}^M \vert\langle \mathcal{U}_i^* \mathbf{x}, f_{i_0j} \rangle \vert^2\\
                                                 &=B\Vert \mathcal{U}_i^*\mathbf{x}\Vert ^2\\
                                                 &=B\Vert \mathbf{x}\Vert ^2
\end{split}                                                 
\]
Hence we have, for any $x \in \mathbb{C}^N$:
\[
\begin{split}
\sum_{i=1}^L\Vert P_ix\Vert ^2&= \sum_{i=1}^L \frac{1}{B}\sum_{j=1}^M \vert \langle P_i\mathbf{x} , f_{ij} \rangle \vert ^2 \\
                                                  &=\sum_{i=1}^L\frac{1}{B} \sum_{j=1}^M\vert \langle \mathbf{x} , f_{ij} \rangle \vert ^2 \\
                                                  &=\frac{1}{B}\sum_{j=1}^M \sum_{i=1}^L\vert \langle \mathbf{x} , f_{ij} \rangle \vert ^2\\
                                                  &=\frac{1}{B}\sum_{j=1}^M A_j \Vert \mathbf{x} \Vert ^2 \quad \quad \\
                                                  &=\frac{\sum_{j=1}^MA_j}{B} \Vert \mathbf{x} \Vert ^2,
 \end{split}                                                                                                 
\]
where $P_i$ is the orthogonal projection on $W_i$. The equality holds since $\{f_{ij}\}_{i=1}^L$ is an $A_j$-tight frame for $\mathbb{C}^N$ for $j=1,\cdots,M$.
\end{proof}   
In the following, we explain the method to construct tight fusion frame based on the Theorem \ref{main1} and  Gabor frames on finite dimensional signals \cite{FiniteFrames}.

To do this, every subspace $W$ can be modeled by a matrix whose rows are an orthonormal basis for $W$. On the other hand, every subspace of dimension $M$ can be represented by a matrix $N \times N$ whose first $M$ rows are an orthonormal basis for $W$, since $\mathbb{C}^{N\times M}$ can be embeded in $\mathbb{C}^{N \times N}$. For example if the subspace $W$ is generated by $\{\mathbf{e}_1, \cdots, \mathbf{e}_{M}\}$, then, the matrix associated to this subspace is as follows: 
\[
\left[ \mathbf{e}_1, \cdots, \mathbf{e}_M,0 , \cdots,0 \right]^*
\]
Moreover, a signal $\mathbf{x}$ of length $N$ can be represenetd by a matrix of $N \times N$ since $\mathbb{C}^N$ can be embeded in $\mathbb{C}^{N \times N}$.
\[
\tilde{X}=\left[ \mathbf{x}, 0 \cdots, 0\right]^*
\]
Based on the notation stated above, we define $\mathbb{C}^{N \times N}$-valued inner product on $\mathbb{C}^{N \times N}$ as follows:
\[
\langle \mathbf{X}, \mathbf{Y} \rangle = \mathbf{X} \mathbf{Y}^*
\]

According to the notions above, if the subspaces $W_i$ of a fusion frame is denoted by a matrices $X_i$, then the fusion frame $\{W_i\}_{i=1}^M$ for $\mathbb{C}^N$ is the same as $\{X_i\}_{i=1}^M$ is a frame for $\mathbb{C}^{N\times N}$, where $\mathbf{x}\in \mathbb{C}^N\subset \mathbb{C}^{N\times N}$. This view point help us to extend several notions and theorems about frame theory to fusion frame theory. For example, the Gabor fusion frame is defined in the following way.

The translation and modulation operators for the space of complex valued square matrix of dimension $N$ are defined as follows: Consider $l \in \mathbb{Z}_N$. The translation operator $\tilde{T}_\ell:\mathbb{C}^{N \times N}\rightarrow \mathbb{C}^{N \times N}$ is defined as follows:
\[
\tilde{T}_\ell \left( \mathbf{e}_1, \cdots, \mathbf{e}_N \right)^*= \left(T_\ell \mathbf{e}_1, \cdots, T_\ell \mathbf{e}_N \right)^*
\]
In fact the translation operator $\tilde{T}_\ell$ alters the position of each row of the matrix $\mathbf{X}$.
The modulation operator $\tilde{M}_\ell:\mathbb{C}^{N \times N}\rightarrow \mathbb{C}^{N \times N}$ is given by
 \[
\tilde{M}_\ell \left( \mathbf{x}_1, \cdots, \mathbf{x}_N \right)^*= \left(M_\ell \mathbf{x}_1, \cdots, M_\ell \mathbf{x}_N \right)^*
\]
Modulation operators are implemented as the pointwise product of each row of the matrix $\mathbf{X}$ with harmonics $e^{-2\pi il \frac{.}{N}}$.
The translation and modulation operator on $\mathbb{C}^{N \times N}$ are unitary operators and the following properties can be concluded 
\[
\left( \tilde{T}_\ell \right)^*= \left( \tilde{T}_\ell \right)^{-1}= \tilde{T}_{N-l} \textrm{and} \left( \tilde{M}_\ell \right)^*= \left( \tilde{M}_\ell \right)^{-1}= \tilde{M}_{N-l}.
\]

The circular convolution of two spaces $\mathbf{X}, \mathbf{Y} \in \mathbb{C}^{N \times N}$ is defined by the convolution of functions, which defined on the space $\mathbb{Z}_N \times \mathbb{Z}_N$ or can be written as:
\[
\mathbf{X} \ast \mathbf{Y}= \left( \sum_{i=0}^{N-1}\mathbf{x}_i \ast \mathbf{y}_{0-i}, \cdots, \sum_{i=0}^{N-1}\mathbf{x}_i \ast \mathbf{y}_{N-1-i} \right)
\]
Hence, if $\mathbf{\tilde{X}}= \left( \mathbf{x}, 0, \cdots, 0 \right)$, the convolution of $\mathbf{\tilde{X}}$ and $\mathbf{Y}$ is given by
\[
\mathbf{\tilde{X}} \ast \mathbf{Y}= \left( \mathbf{x} \ast \mathbf{y}_0, \cdots, \mathbf{x}\ast \mathbf{y}_{N-1} \right)
\]
Moreover, the circular involution or circular adjoint of $\mathbf{X} \in \mathbb{C}^{N\times N} $ is given by 
\[
\mathbf{X}^*= \left( \mathbf{x}_1^*, \cdots, \mathbf{x}_N^*\right)^*
\]
where $\mathbf{x}_1, \cdots, \mathbf{x}_N \in \mathbb{C}^p$ and $\mathbf{x}_i^*\left(\ell \right)= \overline{\mathbf{x}\left( N- \ell \right)}$. Note that the complex linear space $\mathbb{C}^{N\times N}$ equipped with $\ell^1$-norm, the circular convolution and involution defined above is a Banach $*$-algebra.

The unitary discrete Fourier transform of $\mathbf{X} \in \mathbb{C}^{N \times N}$ is defined by
\[
\hat{\mathbf{X}}= \left( \mathcal{F}_N\left( \mathbf{x}_1\right), \cdots, \mathcal{F}_N\left( \mathbf{x}_N\right) \right) 
\]
where $\mathbf{x}_1, \cdots, \mathbf{x}_N \in \mathbb{C}^N$ and the Fourier transform $\mathbf{x}_i$ is given by
\[
\mathcal{F}_N\left(\mathbf{x}_i\right) \left( \ell \right)= \frac{1}{\sqrt{N}} \sum_{k=0}^{N-1} \mathbf{x}_i\left( k\right) \overline{\omega}_{\ell} \left(k\right)=\frac{1}{\sqrt{N}} \sum_{k=0}^{N-1} \mathbf{x}_i\left( k\right) e^{-2\pi i \ell \frac{k}{N}}
\]
The Fourier transform is a unitary operator on the $\mathbb{C}^{N \times N}$ with the Frobenius norm. In fact, for all $\mathbf{X}\in \mathbb{C}^{N \times N}$ : 
\[
\|\langle \hat{\mathbf{X}}, \hat{\mathbf{X}} \rangle \|= \|\langle \mathbf{X}, \mathbf{X} \rangle \|
\]
We also have the following relationships.
\[
\widehat{\tilde{T}_{\ell}\mathbf{X}}= \tilde{M}_{\ell} \hat{\mathbf{X}} \quad \widehat{\tilde{M}_{\ell}\mathbf{X}}= \tilde{T}_{N-\ell}\hat{\mathbf{X}} \quad \hat{\mathbf{X}^*}= \overline{\hat{\mathbf{X}}} \quad \widehat{\mathbf{X} \ast \mathbf{Y}}= \hat{\mathbf{X}}.\hat{\mathbf{Y}}
\]
for $\mathbf{X}, \mathbf{Y} \in \mathbb{C}^{N \times N}$ and $\ell \in \mathbb{Z}_N$. The inverse Fourier formula for $\mathbf{X} \in \mathbb{C}^{N \times N}$ is given by
\[
\mathbf{X}= \left(\mathbf{x}_1, \cdots, \mathbf{x}_N \right)^*=\left( \mathcal{F}_N^{-1}\left( \mathbf{x}_1 \right), \cdots, \mathcal{F}_N^{-1}\left( \mathbf{x}_N \right) \right) ^*
\]
Translation operators are refered as time shift operators and modulation operators are refered as frequency shift operators.
Time-frequency shift operators $\pi \left( k, l\right)$ combines translations by $k$ and modulation by $l$.
\[
\pi \left( k,\ell \right) \mathbf{X}= \tilde{M}_\ell \tilde{T}_k\mathbf{X}
\]

The Gabor Fusion transform $V_{\mathbf{Y}}$ of a signal $\mathbf{x} \in \mathbb{C}^N$ with respect to the window $\mathbf{Y}\in \mathbb{C}^{N\times N}$ is given by
\begin{equation}
V_{\mathbf{Y}}\mathbf{x}\left( k,\ell \right)= \langle \mathbf{x}, \pi \left( k,\ell \right) \mathbf{Y} \rangle= \left( V_{\mathbf{y}_0}\mathbf{x}\left( k,\ell \right), \cdots, V_{\mathbf{y}_{N-1}}\mathbf{x}\left( k,\ell \right) \right)^*
\end{equation}
 Now consider $\mathbf{Y} \in \mathbb{C}^{N \times N}$ and $\Lambda \subset \{ 0, \cdots, N-1 \} \times \{0, \cdots, N-1\}$. The set
\[
\left( \mathbf{Y}, \Lambda \right) =\{ \pi \left(k,\ell\right) \mathbf{Y}\}_{\left(k,\ell\right) \in \Lambda}
\]
is called the Gabor Fusion System which is generated by $\mathbf{Y}$ and $\Lambda$. A Gabor Fusion System which spans $\mathbb{C}^N$ is a fusion frame and is referred to as a Gabor Fusion Frame.
\\Next theorem explains the necessary conditions that the set $\{\tilde{M}_{\ell}\tilde{T}_k \mathbf{Y}\}_{\ell=1, k=1}^{N, N}$ becomes a tight fusion frame.
\begin{theorem}
Assume $\mathbf{x}\in \mathbb{C}^N$  and $\{\mathbf{y}_1,\cdots,\mathbf{y}_M\}$ is a $B$-tight fusion frame for $W_{N,N}=span\{\mathbf{y}_1,\cdots,\mathbf{y}_M\}$. Consider also $W_{k,\ell}=span\{T_kM_\ell\mathbf{y}_j\}_{j=1}^M$ for $k,\ell=1,\cdots,N$. Then, the set $\{W_{k,\ell}\}_{k=1,\ell=1}^{N,N}$ constitutes a $\frac{N\Vert \mathbf{Y}\Vert _2^2}{B}$ tight fusion frame and we have the following equality:
\[
\sum_{k,\ell=0}^{N-1}\Vert P_{k,\ell} \mathbf{x}\Vert^2=\frac{N\Vert \mathbf{Y}\Vert _2^2}{B}\Vert \mathbf{x}\Vert _2^2
\]

\end{theorem}

\begin{proof}
All that has to be done is to verify that $\left\{ \{ T_k M_\ell \mathbf{y}_i \}_{k, \ell = 1}^N \right\}_{i = 1}^M$ satisfies the criteria of proposition \ref{main1}.  First, for a given value of $j$, we have that $\{ T_k M_\ell \mathbf{y_j} \}_{k,\ell = 1}^N$ is a $A_j=N\| \mathbf{y}_j\|^2$ tight frame in $\mathbb{C}^N$ by the elementary Gabor theory (this can be seen the prior section).  It should clear by its nature that the time-frequency shift operator $T_k M_\ell$ is a co-isometry for a set $k, \ell$, since it was mentioned before $T_k$ and $M_{\ell}$ are both unitary operators for every $k, \ell$. Finally, we know by the assumption that $\{ \mathbf{y}_j \}_{j=1}^M$ is $B$-tight on its ambient space $W_{0,0}$.  Seeing that the conditions for the proposition are satisfied, we have the conclusion that $\{ (W_{k,\ell},1) \}_{k,\ell=0}^{N-1}$ is a $\frac{N\Vert \mathbf{Y}\Vert _2^2}{B}$-tight fusion frame on $\mathbb{C}^N$.
\end{proof}



\section{\bf{Gabor Fusion Frames and Phaseless Reconstruction}}
In this section, we are looking for some conditions such that the tight Gabor fusion frame allows phase retrieval. To state these conditions, we provide some theorems should be necessary to explain the main result.  The next lemma shows that if we add a vector to a phaseless retrieval frame, the new frame also allows phaseless retrieval.
\begin{lemma}\label{p1}
Let $\{\phi_i\}_{i=1}^N$ be a frame for $\mathbb{C}^N$ that allows phase reconstruction. If we add a vector $\phi_{N+1}$ to $\{\phi_i\}_{i=1}^N$, then $\{\phi_i\}_{i=1}^{N+1}$, this will also allow phaseless reconstruction. 
\end{lemma}
\begin{proof}
Consider that for $x_1,x_2\in \mathbb{C}^N$, we have $\{\vert \langle x_1,\phi_i \rangle\vert\}_{i=1}^{N+1}=\{\vert \langle x_2,\phi_i \rangle\vert\}_{i=1}^{N+1}$. Hence, we have $\{\vert \langle x_1,\phi_i \rangle\vert\}_{i=1}^{N}=\{\vert \langle x_2,\phi_i \rangle\vert\}_{i=1}^{N}$. So, $x_1=cx_2$ where $\vert c \vert=1$ since $\{\phi_i\}_{i=1}^N$ allows phase retrieval for $\mathbb{C}^N$. Thus $\{\phi_i\}_{i=1}^{N+1}$ also allows phase retrieval. 
\end{proof}
The prior lemma is important in the construction of phase retrieval frames. If we have a phase retrieval frame for $\mathbb{C}^N$, then we can construct a new frame that also allows phase retrieval by adding a vector to the frame vector set. On the other hand, to show the phase retrievability of a frame, it is enough to show that a subset of the frame vectors that spans the ambient space allows phaseless reconstruction.

Next proposition will state the conditions such that a fusion frame is phase retrieval
\begin{proposition}\label{phase}
Let $\{e_i\}_{i=1}^N$ be an orthonormal basis for $\mathbb{C}^N$. Moreover, for every $j=1,\cdots,M$,$\{f_{ij}\}_{i=1}^n$ is a Parseval frame for the subspace $W_j$ generated by these vectors and $f_{ij}$ is the linear sumation of $\{e_i\}_{i=1}^N$. Suppose that $\{f_{ij}\}_{j=1}^M$ for every $i=1,\cdots,n$ is a Parseval frame for $\mathbb{C}^N$ and $\{W_j\}_{j=1}^M$ is a fusion frame and there exists $i_0$ such that $\{f_{i_0j}\}_{j=1}^M$ is a phase retrieval frame for $\mathbb{C}^N$. Then $\{W_j\}_{j=1}^M$ is a phase retrieval fusion frame for $\mathbb{C}^N$ if the matrix $S_{M\times N}$ has a left inverse matrix $V_{N\times M}$ such that
\[
VS=I_{N\times N}.
\]
\end{proposition}
\begin{proof}
To show that there is an injective mapping from  the fusion frame measurements, $\{ \Vert P_{j}\mathbf{x} \Vert_2^2 \}_{j=1}^M$, to the vector $\mathbf{x}$ modulo phase (i.e., the equivalence class $\{ c \mathbf{x} : |c| = 1 \}$), we can just show that we can derive the values of the  frame measurements $\{| \langle \mathbf{x}, f_{i_0j}\rangle |^2 \}_{j=0}^{M}$ from the fusion frame measurements.  We can see this in the following way:

We denote $| \langle \mathbf{x}, e_{i}\rangle |^2=\nu_i$ for $i=1,\cdots,N$. On the other hand
\[
\Vert P_j\mathbf{x}\Vert _2^2=\sum_{i=1}^n | \langle \mathbf{x}, f_{i_0j}\rangle |^2=\sum_{i=1}^N c_{ij} |\langle\mathbf{x}, e_i\rangle |^2,
\]
since $\{f_{ij}\}_{i=1}^n$ is a Parseval frame for $\mathbb{C}^N$ for every $j=1,\cdots,M$ and $f_{ij}$ is the linear summation of $\{e_i\}_{i=1}^N$. We denote $S=\left[c_{ij}\right]_{j=1,i=1}^{M,N}$. Now consider $S\nu$. We will get the following output:
\[
[ \Vert P_{1}\mathbf{x} \Vert_2^2 , \Vert P_{2}\mathbf{x} \Vert_2^2 , \cdots, \Vert P_{M}\mathbf{x} \Vert_2^2 ]^T = S \nu
\]
Since $S$ has a left inverse matrix $V$ and $\{f_{i_0j}\}_{j=1}^M$ is a phase retrieval frame for $\mathbb{C}^N$, we are done.
\end{proof}
The Proposition \ref{phase} has an important role to construct phase retrieval fusion frame based on the phase retrieval frame.

\subsection{A Brief Overview of Circulant Matrices}

We will need to review a few key concepts of circulant matrices before we continue to the next section.

\begin{definition}
\label{CirculantDef}
A circulant matrix is a matrix of the following form:

$$
C=
\begin{bmatrix}
c_0     & c_{n-1} & \dots  & c_{2} & c_{1}  \\
c_{1} & c_0    & c_{n-1} &         & c_{2}  \\
\vdots  & c_{1}& c_0    & \ddots  & \vdots   \\
c_{N-2}  &        & \ddots & \ddots  & c_{N-1}   \\
c_{N-1}  & c_{N-2} & \dots  & c_{1} & c_0 \\

\end{bmatrix}.
$$
\end{definition}

\begin{remark}
We denote the $j^{th}$ division of unity as
$$ \omega_j = \mathrm{exp} \left( \frac{2 \pi i j}{N} \right) $$
\end{remark}

We will need the following theorem; a proof is given in \cite{CirculantMatrix}

\begin{theorem}
\label{CirculantDet}
Let $C$ be an $N \times N$ circulant matrix.  

Then $det(C) = \Pi_{j=0}^{N-1} \left( c_0 + c_1 \omega_j + c_2 \omega_j^2 + \cdots + c_{N-1} \omega_j^{N-1} \right) $.
\end{theorem}

\begin{lemma}
\label{CirculantLemma}
Let $C$ be a matrix as in \ref{CirculantDef} with $c_0, c_1, \ldots, c_{n-1} = 1$ and $c_n, c_{n+1}, \ldots, c_{N+1} = 0$ for some $0 < n < N$.  Then $C$ is singular if and only if there is some value $j$, $1 \le j \le N - 1$, such that $N$ divides into $jn$.
\end{lemma}

\begin{proof}
By \ref{CirculantDet}, we know that $C$ is singular if and only if there is some $j$ where $0 \le j \le N-1$ and $\sum_{k=0}^N c_k \omega_j^k = \sum_{k=0}^{n - 1} \omega_j^k =0$.  We notice that for $j=0$, we have $\sum_{k=0}^{n - 1} \omega_0^k = \sum_{k=0}^{n - 1} 1 = n$, so we will only consider the values $1 \le j \le N - 1$.

Consider $\sum_{k=0}^{n - 1} \omega_j^k$.  The geometric series gives us that this is equal to $\frac{ 1 - w_j^n}{1 - w_j}$; this is zero if and only if $w_j^n = \mathrm{exp} \left( \frac{2 \pi i j n}{N} \right) = 1$.  But this will only happen exactly when $\frac{j n}{N}$ is an integer, that is to say, when $N$ divides into $j n$.

\end{proof}

\subsection{Construction of Gabor Tight Fusion Frame}

In \cite{Bojarovska} the conditions on the window function such that the generated Gabor frame allows phase retrieval are given; we now present a method to produce a phase retrieval Gabor fusion frame.  The following theorem demonstrates the relationship of the phase retrievability of the Gabor fusion frames and the phase retrievability of the frame vectors which spans subspaces.

\begin{theorem}\label{th1}
	Let $\{e_i\}_{i=1}^{N}$ be an orthonormal basis for $\mathbb{C}^N$. Let $ \{f_i\}_{i=1}^{N}$ is a Parseval frame for the $n$-dimensional subspace $W_{0,0} \subset \mathbb{C}^N$ spanned by these vectors and $f_i$ for $i=1,\cdots,n$ is the linear summation of $\{e_i\}_{i=1}^N$ where $\sum_{i=0}^nf_i=\sum_{i=0}^{n_0}e_i$.  Moreover, $W_{k,\ell}= span\ \{T_k M_\ell f_i\}_{i=1}^n$ for $k,\ell=0,1,\cdots,N-1$. If there exists an $i_0$ such that $\{T_k M_\ell f_{i_0}\}_{k,\ell=0}^{N-1}$ is a phase retrieval frame for $\mathbb{C}^N$, then $\{W_{k,\ell}\}_{k,\ell=0}^{N-1}$ is a phase retrieval fusion frame if and only if for all values $1 \le j \le N - 1$, we have that $N$ does not divide into $j n_0$.
\end{theorem}

\begin{proof}
To show that $\{W_{k,\ell}\}_{k,\ell=0}^{N-1}$ is phase retrieval, we display that $\{W_{k,\ell}\}_{k,\ell=0}^{N-1}$ stisfies the conditions of the Proposition \ref{phase}. It is trivial $\{T_k M_\ell e_i\}_{i=1}^n$ is Parseval frame for $W_{k,l}$ for every $k,l=1,\cdots,N-1$. Moreover, there exists $i_0$ such that $\{T_k M_\ell \mathbf{e}_{i_0}\}_{k,\ell=0}^{N-1}$ is a phase retrieval frame.

	Now for $\ell \in\{0,1,\cdots,N-1\}$, consider the vector:
	$$v_{\ell} =[ | \langle \mathbf{x}, T_0 M_{\ell}\mathbf{e}_{1}\rangle |^2 , | \langle \mathbf{x}, T_1 M_{\ell} \mathbf{e}_{1}\rangle |^2, \cdots, | \langle \mathbf{x}, T_{N-1}M_{\ell} \mathbf{e}_{1}\rangle |^2 ]^T,$$
	It is trivial that $\{M_{\ell}e_i\}_{i=1}^N$ is also an orthonormal basis for $\mathbb{C}^N$. Moreover, we have:
	\begin{equation}
	\label{ProjEqn}
	\Vert P_{k,\ell}\mathbf{x} \Vert_2^2=\sum_{i=1}^n \vert \langle \mathbf{x}, T_k M_\ell \mathbf{f}_i\rangle \vert ^2=\sum_{i=1}^N c_i \vert \langle \mathbf{x}, M_\ell T_k \mathbf{e}_i\rangle \vert ^2=\sum_{i=1}^{n_0} c_iv_{l_i}.
	\end{equation}
	Now, consider the operator $S:\mathbb{R}^N \mapsto \mathbb{R}^N$ , 
	where $S$ is the circulant matrix such that the $j^{th}$ row is $T_{j-1} ( [c_1, \cdots, c_{n_0}, 0, \cdots, 0] )$, where the area of support in each row is $n$:
	
	$$S = \begin{bmatrix} c_1 & c_2 & c_{3} & \cdots & c_{n_0-1} & c_{n_0} & 0 & \cdots & 0 \\
	0 & c_1 & c_2 & \cdots & c_{n_0-1} & c_{n_0} & 0 & \cdots & 0 \\
	\vdots & & \cdots & & & & & \cdots & \\
	c_{n_0}& 0 & \cdots &  0  & \cdots & 0  & c_1 & \cdots & c_{n_0-1} \\
	c_{n_0-1} & c_{n_0} & 0 & \cdots & 0 & c_1& c_2 & \cdots & c_{n_0-2}\\
	\vdots & &  & \cdots& & & & \cdots & \vdots \\
	c_2 & c_3 & c_4 & \cdots & c_{n_0-1} & c_{n_0} & 0 & \cdots & c_1\\
	\end{bmatrix}$$
	
	By lemma \ref{CirculantLemma}, it can be seen that $S$ is not singular.
		
	Now by the proposition \ref{phase}, $\{W_{k,l}\}_{k,l=0}^{N-1}$ is phase retrieval.
	
\end{proof}

Theorem \ref{th1} demonstrates the relationship between the phase retrievality of Gabor frame and its associated Gabor fusion frame. In \cite{Bojarovska} the conditions on the window function such that the generated Gabor frame allows phaseless reconstruction are given. Based on Theorem \ref{th1}, we presented a method to produce phase retrieval Gabor fusion frame.

We shall end with a brief example of a Gabor fusion frame that allows phase retrieval, as an application of the prior theorem:

\begin{example}
Consider the orthogonal unit vectors ${e}_1=\mathbf{1}_{\{1,2,4\}}/\sqrt{3}$ and ${e}_2=\mathbf{1}_{\{3\}}$ in the space $\mathbb{C}^7$. By the Proposition 2.2 in \cite{Bojarovska}, $\{T_kM_l\mathbf{e}_1\}_{k,l=0}^6$ is a phase retrieval Gabor frame for $\mathbb{C}^7$. Suppose that $Y_{k,l}=span\ \{T_kM_le_i\}_{i=1}^2$ for $k,l=0,\cdots,6$. Since $e_1$ and $e_2$ are orthogonal so they are tight frame for the subspace $W_{0,0}$. As a result we fullfill the requirements of the Theorem \ref{th1} and the Gabor fusion frame $\{\mathbf{Y}_{k,l}\}_{k,l=0}^6$ allows phaseless reconstruction.
\end{example}


\end{document}